\documentclass[10pt]{amsart}
\usepackage{amsmath,amssymb,amsthm,url}
\usepackage{graphicx}
\usepackage{color}
\newtheorem{theorem}{Theorem}    

\newtheorem{corollary}[theorem]{Corollary}
\theoremstyle{definition}
\newtheorem{definition}{Definition}
\newtheorem{example}{Example}
\newtheorem{remark}{Remark}
\newtheorem*{remark*}{Remark}
\newcommand{\Z}{\mathbb{Z}}

\DeclareMathOperator{\Ord}{Ord}

\title{Ma-Qiu index, presentation distance, and local moves in knot theory}
\author[T.Ito]{Tetsuya Ito}
\address{Department of Mathematics, Kyoto University, Kyoto 606-8502, JAPAN}
\email{tetitoh@math.kyoto-u.ac.jp}

\begin{document}

\begin{abstract}
The Ma-Qiu index of a group is the minimum number of normal generators of the commutator subgroup.
We show that the Ma-Qiu index gives a lower bound of the presentation distance of two groups, the minimum number of relator replacements to change one group to the other. Since many local moves in knot theory induce relator replacements in knot groups, this shows that the Ma-Qiu index of knot groups gives a lower bound of the Gordian distance based on various local moves. In particular, this gives a unified and simple proof of the Nakanishi index bounds of various unknotting numbers, including virtual or welded knot cases.
\end{abstract}

\maketitle

The purpose of this note is to give a unified point of view of the applications of the following quantity in knot theory.

\begin{definition}[Ma-Qiu index]
\label{definition:MQ}
The \emph{Ma-Qiu index} $a(G)$ of a group $G$ is the minimum number of normal generators of the commutator subgroup $[G,G]$.
\end{definition}

The Ma-Qiu index $a(K)$ of a knot $K$ in $S^{3}$ is the Ma-Qiu index of its knot group $G(K)=\pi_1(S^{3} \setminus K)$. In \cite{mq}, it is shown that $a(K) \leq u(K)$ holds, where $u(K)$ is the unknotting number, the minimum number of crossing changes that needed to convert $K$ to the unknot. 

The \emph{Nakanishi index} $m(K)$ of a knot $K$ is the minimum number of generators of the Alexander module of $K$. More generally, for a group $G$, the Nakanishi index $m(G)$ is defined as the minimum number of the generators of the Alexander module, the $\Z[G\slash [G,G]]$-module $[G,G]\slash [[G,G],[G,G]]$. Since $m(G)\leq a(G)$, Ma-Qiu's inequality $a(K) \leq u(K)$ subsumes the famous inequality $m(K) \leq u(K)$ \cite{na}.

Our main observation is that the Ma-Qiu index measures the difference of relators of two groups with the same abelianizations and generating sets.

\begin{theorem}
\label{theorem:main}
Let $G=\langle S \: | \: R \rangle$ and $G'=\langle S \: | \: R' \rangle$ be two (possibly infinite) presentations of groups with the same set of generators. 
Let $R_0 =R \cap R'$, $R_G = R \setminus R_0$, $R_{G'}=R' \setminus R_0$, and let $\pi_G:\langle S \rangle \rightarrow G$ be the projection map. 
If $\pi_G(R_{G'}) \subset [G,G]$ and $H_1(G;\Z) \cong H_1(G';\Z)$ then $a(G') \leq \# R_G + a(G)$.
\end{theorem}
\begin{proof}
 For a subset $Y$ of a group $X$ we denote by $\langle \!\langle Y \rangle\!\rangle$ the normal closure of $Y$ in $X$. 

Assume that $a(G)=k$ so there exists $g_1,\ldots,g_k \in [G,G] \subset G$ such that $\langle \! \langle g_1,\ldots, g_k \rangle \! \rangle =[G,G]$.
We take $r_i \in \pi_G^{-1}(g_i) \subset \langle S \rangle$. Then  
\[ H_1(G;\Z) = G \slash [G,G] = G\slash \langle \! \langle g_1,\ldots, g_k \rangle \! \rangle  =  \langle S \: | \: R_0, R_G, r_1,\ldots,r_k \rangle. \]
Since $\pi_G(R_{G'}) \subset [G,G]$ and $H_1(G';\Z) = H_1(G;\Z)$ it follows that
\begin{align*}
H_1(G';\Z) &= H_1(G;\Z) \\
&= \langle S \: | \: R_0, R_G,  r_1,\ldots,r_k\rangle\\
&= \langle S \: | \: R_0, R_{G'},R_{G},  r_1,\ldots,r_k \rangle \\
& = G' \slash\langle\! \langle \pi_{G'}(R_G), \pi_{G'}(r_1),\ldots,\pi_{G'}(r_k) \rangle \! \rangle
\end{align*}
where $\pi_{G'}: \langle S \rangle \rightarrow G'$ be the projection map.
This shows that 
\[ \langle\! \langle \pi_{G'}(R_G), \pi_{G'}(r_1),\ldots,\pi_{G'}(r_k) \rangle \! \rangle = [G',G']\]
hence $a(G') \leq \# R_G + k$.
\end{proof}

This simple property allows us to deduce many properties and applications of Ma-Qiu index.

First of all, we observe the following upper bound of $a(G)$. For a group $G$ let $r(G)$ be the \emph{rank} of $G$, the minimum number of generators.

\begin{corollary}
\label{cor:rank}
If $G$ is finitely generated, then $a(G)$ is finite. More precisely,
\[ a(G) \leq r(G) + \frac{r(H_1(G;\Z))(r(H_1(G;\Z))-3)}{2} \]
holds.
\end{corollary}
\begin{proof}
Let $r=r(G)$ and $h=r(H_1(G;\Z))$. Take a presentation $\langle S \: | \: R\rangle$ of $G$ with the generating set $S=\{s_1,\ldots,s_r\}$. Let $\pi_H:\langle S \rangle \rightarrow H_1(G;\Z)$ be the projection map.
We take a presentation so that $\pi_H(s_1),\ldots,\pi_H(s_{h})$ generate $H_1(G;\Z)$.
For $k=h+1,\ldots,r$ let $\pi_H(s_{k}) = \sum_{i=1}^{h}c_{k,i}\pi_H(s_{i})$.
Then $H_1(G;\Z)$, the abelianization of $G$, is presented as 
\[
 H_1(G;\Z) = \langle S \: | \: R, [s_i,s_j] \ (1\leq i<j \leq h), s_k^{-1}s_1^{c_{k,1}}\cdots s_{h}^{c_{k,h}} (k= i+1,\ldots,r) \rangle.
\]
By Theorem \ref{theorem:main}, $a(G) \leq \frac{h(h-1)}{2}+(r-h)$.
\end{proof}

In \cite{ky} it is shown that $a(K) \leq r(G(K))-1$. Corollary \ref{cor:rank} generalizes this result for general groups, and in particular for link cases.
\begin{corollary}
For an $\ell$-component link $L$, $a(L) \leq r(G(L)) + \frac{\ell (\ell-3)}{2}$. 
\end{corollary}

To illustrate further applications, we introduce the following notions.
Let $S$ be a generating set of a group $G$. We say that a relator $r$, which we regard as an element of the free group $\langle S \rangle$ generated by the set $S$, is \emph{null-homologous} in $G$ if $\pi_G(r) \in [G,G]$, where $\pi_G: \langle S\rangle \rightarrow G$ is the natural projection.

\begin{definition}[(Null-homologous) relator replacement]
For a presentation $\langle S\: | \: R \rangle$ of a group $G$, the \emph{relator replacement} is an operation that replaces a relator $r \in R$ with another relator $r'$ to get a presentation $\langle S\: | \: R \setminus \{r\}, r' \rangle$ of another group $G'$.
We say that a relator replacement is \emph{null-homologous} if $r$ is null-homologous in $G'$ and $r'$ is null-homologous in $G$.
\end{definition}

When $G$ and $G'$ are related by a sequence of null-homologous relator replacement, $H_1(G;\Z)=H_1(G';\Z)$. 
Conversely, when $G$ and $G'$ are finitely generated, then $G$ and $G'$ are related by a sequence of null-homologous relator replacements, since as we have seen in Corollary \ref{cor:rank} we may convert $G$ to $H_1(G;\Z)$ by adding at most $r(G) + \frac{r(H_1(G;\Z))(r(H_1(G;\Z))-3)}{2}$ null-homologous relators.

\begin{definition}[Presentation distance]
Let $G$ and $G'$ be groups with $H_1(G;\Z)=H_1(G';\Z)$.
The \emph{presentation distance} $d_P(G,G')$ of $G$ and $G'$ is the minimum number of null-homologous relator replacements needed to convert $G$ to $G'$. 
\end{definition}

Theorem \ref{theorem:main} shows that $|a(G)-a(G')| \leq 1$ if $G$ and $G'$ are related by the single null-homologous relator replacement, hence we get the following.

\begin{corollary}
\label{cor:MQ-univ-bound}
For groups $G, G'$ with $H_1(G;\Z) = H_1(G';\Z)$, if $G$ and $G'$ are finitely generated, then 
\[ |a(G)-a(G')| \leq d(G,G') \leq r(G) + r(G') + r(H_1(G;\Z))(r(H_1(G;\Z))-3)\]
\end{corollary}

This corollary allows us to give a unified framework for the Ma-Qiu index bound of various distances in knot theory.

Let $\mathcal{Q}$ be a set consisting of pair of (possibly oriented) tangles. Here by abuse of notation, we will often confuse a tangle $Q$ with its diagram.
A \emph{local move $\mathcal{M}=\mathcal{M}_{\mathcal{Q}}$ defined by $\mathcal{Q}$} is an operation of knots that replaces a tangle $Q$ inside a knot with another tangle $Q'$ for $(Q,Q') \in \mathcal{Q}$. For example, the \emph{crossing change}, the most famous and fundamental move in knot theory, is a local move defined by 
\[\mathcal{Q} =\left \{ \left(
\raisebox{-1.1em}{
\begin{picture}(24,24)
\put(0,2){\line(1,1){24}}
\put(10,16){\line(-1,1){10}}
\put(24,2){\line(-1,1){10}}
\end{picture} }
\raisebox{-0.9em}{,}
\raisebox{-1.1em}{
\begin{picture}(24,24)
\put(0,2){\line(1,1){10}}
\put(24,2){\line(-1,1){24}}
\put(14,16){\line(1,1){10}}
\end{picture}}
\right)
\right\}.\]

For a local move $\mathcal{M}$, the \emph{$\mathcal{M}$-Gordian distance} $d_{\mathcal{M}}(K,K')$ of two knots $K,K'$ is the minimum number of the move $\mathcal{M}$ needed to convert $K$ into $K'$. Here we allow $d_{\mathcal{M}}(K,K')=\infty$.
The $\emph{$\mathcal{M}$-unknotting number}$ $u_{\mathcal{M}}(K,K')$ is defined by $d_{\mathcal{M}}(K,{\sf Unknot})$.

By the van-Kampen theorem or the Wirtinger presentation, when a local move $\mathcal{M}$ has the property that it always sends a knot to a knot, the local move $\mathcal{M}$ often induces null-homologous relator replacements of the knot groups hence the $\mathcal{M}$-Gordian distance $d_{\mathcal{M}}(K,K')$ is bounded below by the (certain constant multiple of) the presentation distance.

Thus Corollary \ref{cor:MQ-univ-bound} shows that the Ma-Qiu index gives lower bounds for $\mathcal{M}$-Gordian distance for many local moves $\mathcal{M}$.
In particular, due to the inequality $m(K) \leq a(K)$ of the Nakanishi index, the Nakanishi index also enjoys the same properties. Thus our argument also gives a unified and simple proof of various Nakanishi index bounds in literature.

Here we give one of the most general local move that contains many local moves as their special case.

An $n$-tangle $(B,T)$ is \emph{trivial} if $(B,T)$ is homeomorphic to $(D^{2}\times [0,1], \{p_1,\ldots,p_n\} \times [0,1])$. An $n$-tangle $(B,T)$ gives a pairing of   $2n$ points $\partial T$. We say that two trivial tangles $(B,T)$ and $(B,T')$ with $\partial T=\partial T'$ \emph{connects the same endpoints} if $T$ and $T'$ give the same paring of $2n$ points.

\begin{definition}[Proper $n$-tangle replacement]
We say that a knot $K'$ is obtained by \emph{proper $n$-tangle replacement} from a knot $K$ if there is a ball $B$ such that $(B,T)=(B, B\cap K)$ is a trivial $n$-tangle, and $K'$ is obtained from $K$ by replacing $(B,T)$ with a different trivial tangle $(B,T')$ so that $T$ and $T'$ connects the same endpoints.
\end{definition}

By van-Kampen theorem, a proper $n$-tangle replacement induces $(n-1)$ null-relator replacements of the knot groups. Therefore we get the following.

\begin{theorem}
\label{theorem:n-replacement}
If a knot $K'$ is obtained from a knot $K$ by a proper $n$-tangle replacement, then $|a(K)-a(K')| \leq n-1$.
\end{theorem}

\begin{example}[Sharp move]
The \emph{sharp move} $\#$ is a local move defined by 
\[
\mathcal{Q} =\left \{ \left(
\raisebox{-1.2em}{
\begin{picture}(35,30)
\put(2,3){\vector(1,1){27}}
\put(5,0){\vector(1,1){28}}
\put(30,0){\line(-1,1){10}}
\put(32,3){\line(-1,1){10}}
\put(15,20){\vector(-1,1){10}}
\put(12,18){\vector(-1,1){10}}
\end{picture} }
\raisebox{-1em}{,}
\raisebox{-1.2em}{
\begin{picture}(35,30)
\put(2,3){\line(1,1){10}}
\put(5,0){\line(1,1){10}}
\put(30,0){\vector(-1,1){28}}
\put(32,3){\vector(-1,1){27}}
\put(23,18){\vector(1,1){10}}
\put(20,20){\vector(1,1){10}}
\end{picture} }
\right)
\right\}.\]
Since it is a proper $4$-tangle replacement, it follows that $3|a(K)-a(K')| \leq d_{\#}(K,K')$ which subsumes the inequality $3|m(K)-m(K')| \leq d_{\#}(K,K')$ \cite{ms}.
\end{example}

A proper $2$-tangle replacement is called a proper rational tangle replacement \cite{ilm,mz}. The \emph{rational unknotting number} $u_{q}(K)$ is the proper $2$-tangle replacement unknotting number, the minimum number of proper rational tangle replacements needed to convert $K$ into the unknot. Theorem \ref{theorem:n-replacement} gives the following refinement of original Ma-Qiu's inequality $a(K) \leq u(K)$.

\begin{corollary}
\[ m(K) \leq a(K) \leq u_{q}(K)\leq u(K)\]
\end{corollary}

This inequality is useful to determine the Nakanishi index of knots.

\begin{example}[Nakanishi index of $12a_{504},12a_{642},12n_{278}$]
According knotinfo \cite{lm}, the Nakanishi of knots $12a_{504},12a_{642},12n_{278}$ are either $1$ or $2$ but their precise values were unknown. 

The knot $12a_{504}$ is the Montesions knot $K(\frac{2}{3},\frac{1}{3},\frac{12}{5})=K(\frac{2}{3},\frac{10}{3},\frac{-3}{5})$. By replacing the rational tangle $\frac{10}{3}$ with the rational tangle $\frac{0}{1}$ we get the unknot (c.f \cite[Theorem 7]{mz}). Since this is a proper rational tangle replacement, $u_q(12n_{278})=1$ hence $m(12a_{504})=a(12a_{504})=1$. 

Similarly, the knot $12a_{642}$ (resp. $12n_{278}$) is the Montesinos knot $K(\frac{1}{3}, \frac{3}{4},\frac{2}{7}) = K(\frac{4}{3},-\frac{1}{4},\frac{2}{7})$ (resp. $K(\frac{2}{3},-\frac{3}{5}, \frac{4}{5})$). By replacing the rational tangle $\frac{4}{3}$ (resp. $\frac{4}{5}$) with $\frac{0}{1}$ we get the unknot hence $u_q(12a_{642})=u_q(12n_{278})=1$. Therefore $m(12a_{642})=m(12n_{278})=1$.
\end{example}

One of an advantage of our argument is that it is algebraic, hence the same argument can be directly applied for wider situations, such as, virtual or welded knot, or, 2-knots. This makes a sharp contrast with geometric method such as Montesinos trick (see \cite{mz} for an application for Montesions trick for rational unknotting number, for example).

For a virtual or welded knot $K$, let $G(K)$ be its knot group, the group obtained by Wirtinger presentation of its diagram. The Ma-Qiu index $a(K)$ is defined as $a(G(K))$.

Although our argument applies to give a lower bound of distances for more complicated moves (for example, virtualized sharp and other moves introduced in \cite{nnsw}), here we illustrate the simplest move.

The \emph{virtualization} is a move of virtual knot diagram defined by 
\[ \mathcal{Q}= \left \{
\left(
\raisebox{-1.1em}{
\begin{picture}(24,24)
\put(0,2){\line(1,1){24}}
\put(10,16){\line(-1,1){10}}
\put(24,2){\line(-1,1){10}}
\end{picture} }
\raisebox{-1em}{,}
\raisebox{-1.1em}{
\begin{picture}(24,24)
\put(0,2){\line(1,1){24}}
\put(24,2){\line(-1,1){24}}
\put(12,14){\circle{6}}
\end{picture}}
\right)\raisebox{-1em}{,} \
\left(
\raisebox{-1.1em}{
\begin{picture}(24,24)
\put(0,2){\line(1,1){10}}
\put(24,2){\line(-1,1){24}}
\put(14,16){\line(1,1){10}}
\end{picture} }
\raisebox{-1em}{,}
\raisebox{-1.1em}{
\begin{picture}(24,24)
\put(0,2){\line(1,1){24}}
\put(24,2){\line(-1,1){24}}
\put(12,14){\circle{6}}
\end{picture} }
\right)
\right\}.\]
A virtual or welded knot diagram $D$ is \emph{$(m,n)$-unknottable} if $D$ can be converted to a diagram of the trivial knot by applying $m$ virtualizations and $n$ crossing changes. From the Wirtinger presentation, it follows that the virtualization induces a single null-homologous relator replacement of the knot groups. Therefore we get the following improvement of the Nakanishi index bound in \cite[Theorem 1]{kkkm}.

\begin{corollary}
If $K$ admits an $(m,n)$-unknottable diagram $D$, then $m(K) \leq a(K)\leq m+n$.
\end{corollary}

Finally, we point out the Ma-Qiu index is also applied to a Dehn surgery distance.

\begin{corollary}
\label{cor:surgery}
Let $M$ and $M'$ be oriented closed 3-manifolds with $H_1(M;\Z) \cong H_1(M',\Z)$. If $M'$ is obtained by a Dehn surgery on a knot $K$ in $M$, $|a(\pi_1(M))- a(\pi_1(M'))| \leq 1$.
\end{corollary}

This is a generalization of a simple observation that when an integral homology sphere $M$ is obtained by a Dehn surgery on a knot in $S^{3}$, $\pi_1(M) (=[\pi_1(M),\pi_1(M)])$ is normally generated by one element. Similarly, by applying Corollary \ref{cor:surgery} for the knot complement, we get the inequality $m(K) \leq a(K) \leq sd(K)$ \cite{na2}.
Here $sd(K)$ is the \emph{surgery description number} of the knot $K$, the minumum number of null-homologous twists (i.e. $\frac{1}{m}$-surgery on unknot $C$ in the knot complement $S^{3} \setminus K$ whose linking number to $K$ is zero) to make $K$ the unknot.

\begin{remark}
The assumption $H_1(M;\Z) = H_1(M';\Z)$ is crucial. 
Let $M$ be the $0$-surgery of the connected sum of $n$ trefoils. Since $M$ is fibered, by \cite{ka} $a(\pi_1(M))= m(\pi_1(M)) = n$. 
\end{remark}

We close this note by pointing out a similarity between the Ma-Qiu index $a(K)$ (or, the Nakanishi index $m(K)$) and the torsion order $\Ord(K)$ of the knot Floer homology. 
In \cite{e2} it is shown that $\Ord(K)$ enjoys the same inequality as Theorem \ref{theorem:n-replacement}, namely, $|\Ord(K)-\Ord(K')| \leq n-1$ holds when $K$ and $K'$ are related by a proper $n$-tangle replacement.

It is also shown that $\Ord(K) \leq br(K)-1$, where $br(K)$ is the bridge index of the knot $K$ \cite{jmz}.
Although $a(K)\leq r(G(K)) \leq br(K)-1$, unlike Ma-Qiu index, $\Ord(K) \leq r(G(K))$ is not true in general. For the $(p,q)$-torus knot $T_{p,q}$ with $p<q$, $\Ord(T_{p,q})=p-1$ \cite{e1} but $r(G(T_{p,q}))=2$. Similarly, for the knot $10_{63}$,  $\Ord(10_{63})=1$ since it is alternating \cite{e1} but $m(10_{63})=a(10_{63})=u_q(10_{63})=u(10_{63})=2$. Thus although there are interesting similarities, the Ma-Qiu index (or the Nakamishi index) and the torsion order are independent.

\section*{Acknowledgements}

The author thank Teruhisa Kadokami for stimulating discussions and encouragements. 
The author is partially supported by JSPS KAKENHI Grant Numbers 19K03490, 21H04428, and 23K03110.

\end{document}